\documentclass[a4paper,11pt]{amsart}

\usepackage{hyperref} 
\usepackage[lite]{amsrefs} 
\usepackage{microtype} 
\usepackage{verbatim} 
\usepackage{color}  
\usepackage{amsmath,amssymb} 
\usepackage[all]{xy} 
\usepackage{enumerate} 

\title[Approximating Novikov--Shubin numbers]{Approximating Novikov--Shubin numbers of virtually cyclic coverings}
\author{Holger Kammeyer}
\address{Institute for Algebra and Geometry\\ Karlsruhe Institute of Technology\\ Germany}
\email{holger.kammeyer@kit.edu}
\urladdr{www.math.kit.edu/iag7/~kammeyer/}

\date{September 2015}

\newtheorem{theorem}{Theorem}
\newtheorem{corollary}{Corollary}
\newtheorem{proposition}{Proposition}
\newtheorem{lemma}{Lemma}

\newtheorem{question}{Question}

\theoremstyle{definition}
\newtheorem{definition}{Definition}

\theoremstyle{remark}

\makeatletter
   \let\c@corollary=\c@theorem
   \let\c@proposition=\c@theorem
   \let\c@lemma=\c@theorem
   \let\c@definition=\c@theorem
   \let\c@remark=\c@theorem
   \let\c@example=\c@theorem
   \let\c@equation=\c@theorem
   \let\c@conjecture=\c@theorem
   \let\c@question=\c@theorem
\makeatother

\newcommand*{\MRref}[2]{ \href{http://www.ams.org/mathscinet-getitem?mr=#1}{MR \textbf{#1}}}
\newcommand*{\arXiv}[1]{ \href{http://www.arxiv.org/abs/#1}{arXiv:\textbf{#1}}}

\newcommand*{\N}{\mathbb N}
\newcommand*{\Z}{\mathbb Z}
\newcommand*{\Q}{\mathbb Q}
\newcommand*{\R}{\mathbb R}
\newcommand*{\C}{\mathbb C}
\newcommand*{\ima}{\textup{i}}

\DeclareMathOperator{\tr}{tr}
\DeclareMathOperator{\rank}{rank}
\DeclareMathOperator{\res}{res}
\DeclareMathOperator{\im}{im}
\DeclareMathOperator{\lcm}{lcm}

\newcounter{commentcounter}

\newcommand{\ignore}[1]{}

\begin{document}

\begin{abstract}  
We assign real numbers to finite sheeted coverings of compact CW complexes designed as finite counterparts to the Novikov--Shubin numbers.  We prove an approximation theorem in the case of virtually cyclic fundamental groups employing methods from Diophantine approximation.
\end{abstract} 

\maketitle

\section{Introduction}

\noindent Let \(X\) be a compact connected CW complex, let \(\widetilde{X}\) be the universal covering and let \(\overline{X}\) be a finite sheeted Galois covering.  In this paper we will define the \emph{alpha numbers} \(\alpha_p(\overline{X}) \in \R\) in terms of the singular value decomposition of the \(p\)-th cellular differential of \(\overline{X}\).  Intuitively, the definition of \(\alpha_p(\overline{X})\) in terms of singular values mimics the definition of \emph{Novikov--Shubin numbers} \(\alpha^{(2)}_p(\widetilde{X})\) in terms of spectral distribution functions.  A natural question then asks whether the Novikov--Shubin numbers can be recovered asymptotically from the net of alpha numbers \((\alpha_p(\overline{X_i}))_{i \in F}\) of all finite Galois coverings of \(X\).  We show that the answer is yes if the fundamental group contains a cyclic subgroup of finite index.

\begin{theorem} \label{thm:approxvcycspaceversion}
Suppose \(\pi_1(X)\) is virtually cyclic and \(\alpha_p(\widetilde{X}) < \infty^+\).  Then
  \[ \alpha^{(2)}_p(\widetilde{X}) = \limsup\limits_{i \in F} \alpha_p(\overline{X}_i). \]
\end{theorem}

\noindent Moreover, we construct a CW complex \(X\), obtained from \(S^1 \vee S^2\) by attaching one 3-cell, such that \(\alpha_3^{(2)}(\widetilde{X}) = 1\) but \(0 < \liminf_{i \in F} \alpha_3(\overline{X}_i) \le \frac{1}{2}\).
\subsection{The definition of alpha numbers}
To construct the numbers \(\alpha_p(\overline{X}_i)\), we will have to take a close look on the definition of Novikov--Shubin numbers.   In doing so, let us go over from spaces to matrices which seem to form the appropriate setting for the approximation theory of \(L^2\)-invariants.

Let \(G\) be a countable, discrete group and let \(A \in M(r,s; \C G)\) be a matrix inducing the right multiplication operator \(r_{AA^*}^{(2)} \colon (\ell^2 G)^r \rightarrow (\ell^2 G)^r\) given by \(x \mapsto xAA^*\).   Here the matrix \(A^*\) is obtained from \(A\) by transposing and applying the canonical involution \((\sum \lambda_g g)^* = \sum \overline{\lambda_g} g^{-1}\) to the entries.  Let \(\{ E^{AA^*}_\lambda \}_{\lambda \ge 0}\) be the family of equivariant spectral projections obtained from \(r^{(2)}_{A A^*}\) by Borel functional calculus, \(E^{A A^*}_\lambda = \chi_{[0,\lambda]}(r^{(2)}_{AA^*})\), where \(\chi_{[0,\lambda]}\) is the characteristic function of the interval \([0,\lambda]\).  Recall that the group von Neumann algebra \(\mathcal{N}(G)\) of \(G\) comes endowed with a canonical finite, faithful, normal trace \(\tr_{\mathcal{N}(G)}\) which extends diagonally to equivariant operators of \((\ell^2 G)^r\).

\begin{definition}
  The function \(F_A \colon [0,\infty) \rightarrow [0,\infty)\) given by \(\lambda \mapsto \tr_{\mathcal{N}(G)}E^{AA^*}_{\lambda^2}\) is called the \emph{spectral distribution function} of the matrix \(A\).  The \emph{upper Novikov--Shubin number} of \(A\) is given by
\[ \overline{\alpha}^{(2)}(A) = \limsup_{\lambda \rightarrow 0^+} \frac{\log(F_A(\lambda) - F_A(0))}{\log \lambda} \in [0,\infty] \]
unless \(F_A(\lambda) = F_A(0)\) for some \(\lambda > 0\) in which case we set \(\alpha^{(2)}(A) = \infty^+\).  The \emph{lower Novikov--Shubin number} \(\underline{\alpha}^{(2)}(A)\) of \(A\) is defined similarly with ``\(\liminf\)'' in place of ``\(\limsup\)''.  We say that \(A\) has the \emph{limit property} if \(\overline{\alpha}^{(2)}(A) = \underline{\alpha}^{(2)}(A)\).  In this case we simply call this common value the \emph{Novikov--Shubin number} \(\alpha^{(2)}(A)\).
\end{definition}

\noindent The formal symbol ``\(\infty^+\)'' indicates a spectral gap at zero.  We adopt the convention that \(c < \infty < \infty^+\) for all \(c \in [0,\infty)\).  Novikov--Shubin numbers thus capture the polynomial growth rate near zero of the spectral distribution function \(F_A\).  More precisely, if there are constants \(C, d, \varepsilon > 0\) such that \(C^{-1} \lambda^d \le F_A(\lambda) - F_A(0) \le C \lambda^d\) for \(\lambda \in [0, \varepsilon)\), then \(A\) has the limit property and \(\alpha^{(2)}(A) = d\).  We should say that while the distinction between upper and lower Novikov--Shubin numbers is already contained in \cite{Gromov-Shubin:vonNeumannSpectra}, the (somewhat arbitrary) decision that \(\alpha^{(2)}(A)\) should mean \(\underline{\alpha}^{(2)}(A)\) has become accepted in the literature.

    Now let \(G\) be \emph{residually finite} meaning there exists a \emph{residual system} \((G_i)_{i \in I}\), an inverse system of finite index normal subgroups directed by inclusion over a directed set \(I\) with trivial total intersection.  We obtain matrices \(A_i \in M(r,s; \C (G/G_i))\) from \(A\) by applying the canonical projections \(\C G \rightarrow \C(G/G_i)\) to the entries.  Set \(n_i = [G \colon G_i]\).  Then the group algebra \(\C(G/G_i)\) embeds as a subalgebra of \(M(n_i, n_i; \C)\) by means of the left regular representation of the finite group \(G/G_i\).  Accordingly, we can view \(A_i\) as lying in \(M(r n_i, s n_i; \C)\).  So we can consider the positive \emph{singular values}
\[ \sigma_1(A_i) \ge \cdots \ge \sigma_{r_i}(A_i) > 0 \]
of \(A_i\) given by \(\sigma_j(A_i) = \sqrt{\lambda_{j,i}}\) where the \(\lambda_{j,i}\) are the positive eigenvalues of \(A_iA_i^*\) in non-ascending order and \(r_i = \rank_\C A_i\).  We denote the \emph{multiplicity} of \(\sigma_j(A_i)\) as \(m_j(A_i) = \dim_\C \ker(A_i A_i^* - \lambda_{j,i})\) and set \(m_{r_i + 1}(A_i) = \dim_\C\ker(A_i A_i^*)\).  With this data, the spectral distribution function \(F_{A_i}\) can be described as a monotone, right continuous step function with jumps at the singular values \(\sigma_j(A_i)\) and jump size \(\frac{m_j(A_i)}{n_i}\).  It is known that these step functions approximate the spectral distribution function \(F_A\).  More precisely,
\[ F_A(\lambda) = \lim_{\delta \rightarrow 0^+} {\textstyle \limsup\limits_{i \in I}}\, F_{A_i}(\lambda + \delta) = \lim_{\delta \rightarrow 0^+} {\textstyle \liminf\limits_{i \in I}}\, F_{A_i}(\lambda + \delta) \]
as is proven in \cite{Lueck:Approximating}*{Theorem~2.3.1} for residual chains (when \(I\) is totally ordered), the proof for residual systems being similar.

So we might want to think about the values \(F_{A_i}(\sigma_j(A_i)) = \sum_{k \ge j} \frac{m_k(A_i)}{n_i}\) as experimental samples of the function of interest \(F_A\).  To extract the growth rate of \(F_A\) from these samples we do what every physicist would do: we measure the slope of the regression line through the doubly logarithmic scatter plot of the samples.  The sample that is most valuable for our purposes is given by the first positive singular value \(\sigma^+(A_i) = \sigma_{r_i}(A_i)\) with multiplicity \(m^+(A_i) = m_{r_i}(A_i)\).

\begin{definition}
The \emph{alpha number} of a nonzero \(A_i \in M(r,s; \C(G/G_i))\) is
\[ \alpha(A_i) = \frac{\log\frac{m^+(A_i)}{[G \colon G_i]}}{\log \sigma^+(A_i)} \in \R. \]
\end{definition}

\noindent Choosing the first positive singular value in the above definition serves a double purpose.  Firstly, this makes sure that the growth behavior close to zero is reflected because \(\lim_i \sigma_+(A_i) = 0\) whenever \(\alpha^{(2)}(A) < \infty^+\).  Secondly, since therefore \(\log \sigma_+(A_i)\) tends to \(-\infty\), the alpha number ultimately measures the slope of the line through the origin which is parallel to the regression line and hence has the same slope.  Finally note that the embedding \(\C (G/G_i) \subset M(n_i,n_i; \C)\) as a subalgebra is unique up to conjugating with a permutation matrix and a diagonal matrix with entries \(\pm 1\).  Any two resulting embeddings \(M(r, s; \C (G/G_i)) \subset M(r n_i, s n_i; \C)\) are thus conjugate by a unitary transformation which leaves the singular value decomposition unaffected.  This shows that the alpha number is well-defined.

\subsection{Approximating Novikov--Shubin numbers by alpha numbers}

The canonical example of a residual system is the \emph{full residual system} \((G_i)_{i \in F}\) of \emph{all} finite index normal subgroups of \(G\).  We ask the following question.

\begin{question} \label{question:approxmatrixversion}
  Let \(G\) be a residually finite group, let \(\Q \subset F \subset \C\) be a field and let \(A \in M(r,s; F G)\).  Suppose that \(\overline{\alpha}^{(2)}(A) < \infty^+\).  Is it true that
  \begin{enumerate}[(a)]
  \item \label{item:sup} \(\overline{\alpha}^{(2)}(A) = \limsup_{i \in F} \alpha(A_i)\)?
  \item \label{item:inf} \(\underline{\alpha}^{(2)}(A) = \liminf_{i \in F} \alpha(A_i)\)?
  \end{enumerate}
\end{question}

\noindent In this paper we answer Question~\ref{question:approxmatrixversion} for virtually cyclic groups.

\begin{theorem} \label{thm:approxvcycmatrixversion}
Let \(G\) be a virtually cyclic group and let \(\Q \subset F \subset \C\) be an arbitrary field.  Then the answer to Question~\textup{\ref{question:approxmatrixversion}\,\eqref{item:sup}} is positive and the answer to Question~\textup{\ref{question:approxmatrixversion}\,\eqref{item:inf}} is negative.
\end{theorem}

\noindent We remark that the related \emph{approximation conjecture for Fuglede--Kadison determinants} \cite{Lueck:Survey}*{Conjecture~6.2} is likewise only known for virtually cyclic groups \cite{Schmidt:DynamicalSystems}.  Though the class of groups is small, the proof of Theorem~\ref{thm:approxvcycmatrixversion} is nontrivial and requires number theoretic input.  Here also lies the reason for the symmetry breaking answer which at first glance might come as a surprise.  It is the existence of infinitely many good rational approximations to a given irrational number which tears the lower limit apart from the upper one.  But for virtually cyclic \(G\) it is easy to see that every \(A \in M(r,s;\C G)\) has the limit property.  So even for virtually cyclic groups the equality \(\alpha^{(2)}(A) = \limsup_{i \in F} \alpha(A_i)\) cannot be improved to \(\alpha^{(2)}(A) = \lim_{i \in F} \alpha(A_i)\).  However, for \(F = \Q\) we can show that \(\liminf_{i \in F} \alpha(A_i)\) is always positive as a consequence of a result in transcendence theory.  We will discuss this in a moment but first let us return from matrices to spaces and explain that the case \(F = \Q\) of Theorem~\ref{thm:approxvcycmatrixversion} gives Theorem~\ref{thm:approxvcycspaceversion} and the example below it.

Let \(X\) be a connected finite CW complex with \(G = \pi_1(X)\) residually finite.  Choosing a cellular basis of \(X\) gives rise to an isomorphism that identifies the \(p\)-th cellular chain module \(C_p(\widetilde{X})\) of the universal covering with the standard left \(\Z G\)-module \((\Z G)^{N_p}\).  Here \(N_p\) is the number of \(p\)-cells of \(X\) or, equivalently, the number of \(G\)-equivariant \(p\)-cells of the \(G\)-CW complex \(\widetilde{X}\).  Under this isomorphism the \(G\)-equivariant differential \(d_p \colon C_p(\widetilde{X}) \rightarrow C_{p-1}(\widetilde{X})\) of the chain complex \(C_*(\widetilde{X})\) is represented by right multiplication with a matrix \(A(\widetilde{X}, p) \in M(N_p,N_{p-1}; \Z G)\).  We define the \emph{\(p\)-th Novikov--Shubin number} of \(\widetilde{X}\) as \(\alpha^{(2)}_p(\widetilde{X}) = \alpha^{(2)}(A(\widetilde{X}, p))\).  Note that in \cite{Lueck:L2Invariants}*{Definition~2.16, p.\,81} one restricts the induced operator \(d_p \colon \ell^2(G)^{N_p} \rightarrow \ell^2 (G)^{N_{p-1}}\) to the orthogonal complement of \(\im d_{p+1}\) to make sure the spectral distribution function takes the value \(b^{(2)}_p(\widetilde{X})\) at zero.  For the Novikov--Shubin numbers this is of course irrelevant. 

Given a finite index normal subgroup \(G_i \subset G\) we can construct the finite covering space \(\overline{X}_i\) with deck transformation group \(G/G_i\) as \(G_i \backslash \widetilde{X}\).  The chosen cellular basis of \(X\) identifies \(C_p(\overline{X}_i) \cong (\Z (G/G_i))^{N_p}\) and the differential \(d^i_p \colon C_p(\overline{X}_i) \rightarrow C_{p-1}(\overline{X}_i)\) is thus represented by right multiplication with a matrix \(A(\overline{X}_i, p)\) which coincides with the matrix \(A(\widetilde{X},p)_i\) obtained from \(A(\widetilde{X}, p)\) by applying the canonical projection \(\Z G \rightarrow \Z (G/G_i)\) to the entries.  We define the \emph{\(p\)-th alpha number} of \(\overline{X}_i\) as \(\alpha_p(\overline{X}_i) = \alpha(A(\overline{X}_i, p))\).  Both Novikov--Shubin numbers and alpha numbers are well-defined because the isomorphisms \(C_p(\widetilde{X}) \cong (\Z G)^{N_p}\) and \(C_p(\overline{X}_i) \cong (\Z (G/G_i))^{N_p}\) are unique up to unitaries.

With these definitions it is immediate that Theorem~\ref{thm:approxvcycmatrixversion} implies Theorem~\ref{thm:approxvcycspaceversion}.  It is moreover well-known that matrices in \(M(r, s; \Z G)\) can be realized as cellular differentials of \(G\)-CW complexes, compare \cite{Lueck:L2Invariants}*{Lemma~10.5, p.\,371}.  In this way the counterexample we will construct for Question~\ref{question:approxmatrixversion}\,\eqref{item:inf} translates to the example mentioned below Theorem~\ref{thm:approxvcycspaceversion}.

\subsection{The role of the coefficient field}

This realization of matrices over \(\Z G\) as differentials of based \(G\)-CW complexes is why Theorem~\ref{thm:approxvcycspaceversion} is actually equivalent to (a positive answer to) Question~\ref{question:approxmatrixversion}\,\eqref{item:sup} for \(F = \Q\).  Similarly, the aforementioned determinant approximation conjecture \cite{Lueck:Survey}*{Conjecture~6.2} is formulated for coefficients in \(\Q\).  It is remarkable that for coefficients in \(\C\) the statement of the determinant approximation conjecture is wrong, even in the case of a \((1 \times 1)\)-matrix over \(\C[\Z]\), see \cite{Lueck:L2Invariants}*{Example~13.69, p.\,481}.  This is just one instance showing that coefficients matter for approximation questions.  In the ``topological case'' \(F = \Q\), there are results in the theory of linear forms in (two) logarithms which are of value to us.  They allow at least the conclusion that \(\liminf_{i \in F} \alpha(A_i)\) is positive, as it should be, because so is every \(\alpha^{(2)}(A)\).

\begin{theorem} \label{thm:liminfpositive}
Let \(G\) be a virtually cyclic group and let \(A \in M(r, s; \Q G)\) with \(\alpha^{(2)}(A) < \infty^+\).  Then \(\liminf_{i \in F} \alpha(A_i) > 0\).
\end{theorem}

\noindent For Theorem~\ref{thm:approxvcycspaceversion} this says that while it can happen that \(\liminf_{i \in F} \alpha_p(\overline{X}_i) < \limsup_{i \in F} \alpha_p(\overline{X}_i)\), at least we have \(\liminf_{i \in F} \alpha_p(\overline{X}_i) > 0\).  In fact, the number theory involved gives something stronger than Theorem~\ref{thm:liminfpositive}, namely the existence of some \(D > 0\) such that \(\liminf_{i \in F} \alpha(A_i) \ge \frac{\alpha^{(2)}(A)}{D+1}\) together with some explicit bounds for the constant \(D\) in terms of degree and height of a certain polynomial associated with \(A\).  For the precise statement see Corollary~\ref{cor:boundonliminf}.

\subsection{Outline and organization of the paper}

Our proofs of Theorem~\ref{thm:approxvcycmatrixversion} and Theorem~\ref{thm:liminfpositive} rely on methods from Diophantine approximation and transcendence theory.  Since these are topics that tend to fall short in a typical topologist's curriculum, we give a brief recap in Section~\ref{section:preliminaries} and recall the theorems of Dirichlet, Kronecker, Gelfond--Schneider and a baby version of Baker's theorem.  We also fix the terminology we use in the context of nets.

In Section~\ref{section:onebyone} we start with the proof of Theorem~\ref{thm:approxvcycmatrixversion}.  As a warm-up we consider the case of the easiest polynomial \(p(z) = z-1\) and show that Dirichlet's theorem easily answers Question~\ref{question:approxmatrixversion}\,\eqref{item:inf} in the negative.  To answer Question~\ref{question:approxmatrixversion}\,\eqref{item:sup} affirmatively, we then move on with the case of a \((1 \times 1)\)-matrix over the group ring \(\C[\Z]\).  It turns out that again one runs into a problem of Diophantine approximation: Can one find a sequence of regular \(i\)-gons whose vertices are far away from given elements of the unit circle?  Solving this problem amounts to understanding how the rational dependency of coordinates of a torus point determines the closure of its \(\Z\)-orbit.  This is what Kronecker's theorem accomplishes.

In Section~\ref{section:rbys} we perform the passage to \((r \times s)\)-matrices over \(\C[\Z]\).  The methods are singular value inequalities and another simple but effective tool that is widely employed in Diophantine approximation:  the pigeon hole principle.

Section~\ref{section:virtuallycyclic} reduces the general case of a virtually cyclic group to the case of the group \(\Z\) and thereby finishes the proof of Theorem~\ref{thm:approxvcycmatrixversion}.

Finally, Section~\ref{section:liminfpositive} discusses the case of rational coefficients.  The little Baker theorem and thus the theory of bounding linear forms in (two) logarithms is what allows in this case the conclusion of Theorem~\ref{thm:liminfpositive}.  

\subsection{Acknowledgements}

I am indebted to Yann Bugeaud, Wolfgang L\"uck, Malte Pieper, Henrik R\"uping, Roman Sauer and Thomas Schick for helpful conversations.

\section{Preliminaries}
\label{section:preliminaries}

\subsection{Some facts from Diophantine approximation} \label{subsection:diophantine}

For a real number \(x\) let \(\|x\|\) denote the distance to the closest integer.  It is easy to see that the usual triangle equality \(\| x + y \| \le \|x\| + \|y\|\) holds.  From this it follows that \(\|nx\| \le \lvert n \rvert \|x\|\) for any \emph{integer} \(n\).  Dirichlet famously concluded the following result from the pigeon hole principle.

\begin{theorem}[Dirichlet, {\raise.17ex\hbox{$\scriptstyle\sim$}}1840] \label{thm:Dirichlet}
  Given real numbers \(l_1, \ldots, l_u\) and a natural number \(N\), there is \(1 \le q \le N\) such that \(\|q l_i\| \le N^{-\frac{1}{u}}\) for all \(i = 1, \ldots, u\).
\end{theorem}

\noindent Dirichlet's theorem will be key for constructing a counterexample to Question~\ref{question:approxmatrixversion}\,\eqref{item:inf} in Section~\ref{section:onebyone}.  We are moreover interested in an inhomogeneous variant of this problem of simultaneous Diophantine approximation:  If additionally real numbers \(x_1, \ldots, x_u\) and \(\varepsilon > 0\) are given, does there exist \(q \in \Z\) with \(\|ql_i - x_i\| < \varepsilon\) for all \(i = 1, \ldots, u\)?  The answer cannot be an unconditional ``yes'' because there might be integers \(A_1, \ldots, A_u\) with the property that the linear combination \(\sum_{i=1}^u A_i l_i\) is an integer as well.  If the desired conclusion held true, we would get
\[ \| A_1 x_1 + \cdots + A_u x_u \| = \| A_1 (q l_1 - x_1) + \cdots + A_u (q l_u - x_u) \| \le (\lvert A_1 \rvert + \cdots + \lvert A_u \rvert) \varepsilon \]
which says that \(\sum_{i=1}^u A_i x_i\) is an integer, too.  The good news is that this necessary condition is also sufficient.

\begin{theorem}[Kronecker, 1884] \label{thm:Kronecker}
  Let \(l_1, \ldots, l_u\) and \(x_1, \ldots, x_u\) be real numbers.  The following are equivalent:
  \begin{enumerate}[(i)]
  \item For every \(\varepsilon > 0\) there is \(q \in \Z\) such that \(\|q l_i - x_i\| < \varepsilon\) for \(i = 1, \ldots, u\).
  \item For every \(u\)-tuple \((A_1, \ldots, A_u) \in \Z^u\) with the property that \(\sum_{i=1}^u A_i l_i\) is an integer, the linear combination \(\sum_{i=1}^u A_i x_i\) is an integer as well.
    \end{enumerate}
\end{theorem}

\noindent A proof can be found in \cite{Cassels:Diophantine}*{Theorem~IV, p.\,53}.  We remark that Kronecker's theorem is usually given in a slightly more general version where the real numbers \(l_i\) are replaced by linear forms but as of now we do not need this.  Kronecker's theorem will become handy for understanding torus orbits in Section~\ref{section:onebyone}.

\begin{theorem}[Gelfond--Schneider, 1934] \label{thm:Gelfond-Schneider}
Let \(\alpha_1, \alpha_2 \in \overline{\Q}\) be different from \(0\) and~\(1\) such that (some fixed values of) \(\log \alpha_1\) and \(\log \alpha_2\) are linearly independent over \(\Q\).  Then \(\log \alpha_1\) and \(\log \alpha_2\) are linearly independent over \(\overline{\Q}\).
\end{theorem}

\noindent This theorem has the equivalent formulation that for \(\alpha_1, \alpha_2\) as above and additionally \(\alpha_2\) irrational, any value of \(\alpha_1^{\alpha_2}\) is transcendental.  As such, it yields the positive answer to Hilbert's seventh problem.  For applications to Diophantine equations not only the nonvanishing of the \emph{linear form in two logarithms}
\[ \Lambda = b_1 \log \alpha_1 + b_2 \log \alpha_2 \]
is important but also explicit lower bounds on \(\Lambda\) in terms of the heights and degrees of \(b_1, b_2 \in \overline{\Q}\) are relevant.  For our purposes it is enough to consider the special case where \(b_1\) and \(b_2\) are rational integers.

\begin{theorem} \label{thm:littlebaker}
Let \(\alpha_1, \alpha_2 \in \overline{\Q}\) be different from \(0\) and~\(1\) and let \(b_1, b_2\) be rational integers such that \(\Lambda \neq 0\).  Set \(B = \max \{\lvert b_1 \rvert, \lvert b_2 \rvert\}\).  Then there is a constant \(D\) depending only on the heights and degrees of \(\alpha_1\) and \(\alpha_2\) such that
  \[\lvert \Lambda \rvert > B^{-D}. \]
\end{theorem}

\noindent It is hard to track down where exactly in the involved history of bounding logarithms in linear forms the theorem in this formulation was included for the first time.  Gelfond already gave the weaker estimate \(\lvert \Lambda \vert > C e^{-(\log B)^\kappa}\) with improvements on the constant \(\kappa \) over two decades \citelist{\cite{Gelfond:first} \cite{Gelfond:second} \cite{Gelfond:third}}.  But the above theorem is definitely a special case of Baker's celebrated theorem from 1966-1967, see \cite{Baker:LinearForms}*{Theorem~2} for a strong version and information on the constant \(D\).  Let us refer to any \(D = D(\alpha_1, \alpha_2) \ge 1\) satisfying the inequality of the theorem as a \emph{Baker constant} of the pair \((\alpha_1, \alpha_2)\).  Theorem~\ref{thm:littlebaker} will be crucial for the proof of Theorem~\ref{thm:liminfpositive} in Section~\ref{section:liminfpositive}.

\subsection{Nets and cluster points} \label{subsection:nets}

The finite index normal subgroups of a group and thereby the finite Galois coverings of a space are natural examples of \emph{directed sets}.  A set \(I\) is called \emph{directed} if it comes with a reflexive, transitive binary relation ``\(\le\)'' such that any two elements \(a, b \in I\) have a common \emph{upper bound} \(c \in I \) with \(a \le c\) and \(b \le c\).  A function from a directed set \((I, \le)\) to a topological space \(X\) is called a \emph{net} in \(X\).  If \((x_i)_{i \in I}\) is a net in \(X\), then a point \(c \in X\) is called a \emph{cluster point} if for every neighborhood \(U\) of \(c\) and for every \(i \in I\) there exists \(j \ge i\) with \(x_j \in U\).  The set of cluster points is closed.  In the special case \(X = \R\) we define \(\limsup_{i \in I} x_i\) and \(\liminf_{i \in I} x_i\) as the largest and the smallest cluster point, respectively.  Here, we also allow the values \(\pm \infty\) as cluster points in the natural way, so that both \(\limsup_{i \in I} x_i\) and \(\liminf_{i \in I} x_i\)  are guaranteed to exist.  If the latter two are equal, we say the net is \emph{convergent} and write \(\lim_{i \in I} x_i\) for the common value.  Alternatively, we clearly have the description
\[ \liminf_{i \in I} x_i = \sup_{i \in I} \inf_{i \le j} x_j \quad \text{and} \quad \limsup_{i \in I} x_i = \inf_{i \in I} \sup_{i \le j} x_j. \]

For the set of natural numbers \(\N\) we will have occasion to deal with two different directions.  One is the usual total order ``a \(\le\) b'' in which all the above notions reduce to the familiar ones from sequences.  The other is divisibility ``\( a \mid b\)'' and arises when we identify \(\N\) with the full residual system \(F\) of the group \(\Z\).  We should clarify the relation between the resulting upper and lower limits in order to dispel any possible confusion from the very start.

\begin{lemma} \label{lemma:sequencenet}
Let \(a \colon \N \rightarrow \R\) be a function which we interpret either as the sequence \((a_i)_{i \ge 0}\) or as the net \((a_i)_{i \in F}\).  Then
\[ \liminf_{i \rightarrow \infty} a_i \le \liminf_{i \in F} a_i \le \limsup_{i \in F} a_i \le \limsup_{i \rightarrow \infty} a_i \]
where each inequality can be strict.
\end{lemma}

\begin{proof}
Let \(c \in \R\) be a cluster point of the net \((a_i)_{i \in F}\).  By definition this means that for all \(\varepsilon > 0\) and for all \(k \in F = \N\) there is \(l \in \N\) such that \(\lvert a_{kl} - c \rvert < \varepsilon\).  In particular, we obtain a subsequence \((a_{i_k})_{k \ge 0}\) of \((a_i)_{i \ge 0}\) which converges to \(c\).  Thus any cluster point of the net \((a_i)_{i \in F}\) is a cluster point of the sequence \((a_i)_{i \ge 0}\).  This gives the two outer inequalities of the lemma.

Consider the example \(a_i = (-1)^i\).  Then the leftmost inequality is strict for \((a_i)\) and the rightmost inequality is strict for \((-a_i)\).  To see that the middle inequality can be strict, consider \(a_i = (-1)^{N_i}\) where \(N_i\) is the number of prime factors of \(i\).
\end{proof}

\section{The case of a single Laurent polynomial} \label{section:onebyone}

\noindent In this section we give a proof of Theorem~\ref{thm:approxvcycmatrixversion} for \(r = s = 1\).  Consider an element \(A \in M(1,1; \C[\Z])\).  The full residual system is given by \(G_i = i\Z\) for \(i \in F = \N\) directed by divisibility.  We identify the group ring \(\C[\Z]\) with the ring of Laurent polynomials \(\C[z, z^{-1}]\).  Moreover, Fourier transform identifies the Hilbert space \(\ell^2(\Z)\) with \(L^2(S^1, \mu)\), the space of square integrable complex valued functions on the unit circle with respect to the probability Haar measure~\(\mu\), factoring out those function which vanish almost everywhere.

\subsection{Two examples}  \label{subsection:twoexamples}  Let us sneak up on the proof by considering the first nontrivial case \(A = (p(z))\) with \(p(z) = z-1\).    The operator \(r^{(2)}_{AA^*}\) is then given by multiplying functions with \(|z-1|^2\).  We have \(\alpha^{(2)}(A) = 1\) as can be seen from the proof of \cite{Lueck:L2Invariants}*{Lemma~2.58, p.\,101}.  By finite Fourier transform, the matrices \(A_iA_i^* \in M(1,1;\C[\Z / i\Z]) \subset M(i,i; \C)\) are diagonal with entries \(|\zeta_i^k - 1|^2\) where \(\zeta_i\) is one of the two primitive \(i\)-th roots of unity that enclose the smallest angle with \(1 \in \C\), where \(k = 0, \ldots, i-1\) and say \(i \ge 3\).  Thus we have \(\sigma^+(A_i) = |\zeta_i - 1| = 2 \sin(\frac{\pi}{i})\) and \(m^+(A_i) = 2\).  By L'H\^opital's rule and substituting \(x = \frac{\pi}{i}\) the ordinary limit of the sequence \((\alpha(A_i))_{i \ge 0}\) is
\[ \lim_{i\rightarrow\infty} \alpha(A_i) = \lim_{i\rightarrow\infty} \frac{\log\left(\frac{2}{i}\right)}{\log\left(2\sin\left(\frac{\pi}{i}\right)\right)} = \lim_{i\rightarrow\infty} \frac{i \tan\left(\frac{\pi}{i}\right)}{\pi} = \lim_{x \rightarrow 0^+} \frac{\tan(x)}{x} = 1. \]
By Lemma~\ref{lemma:sequencenet} the net \((\alpha(A_i))_{i \in F}\) has limit \(\lim_{i \in F} \alpha(A_i) = 1\) as well.  So in this simplest possible case of Question~\ref{question:approxmatrixversion} the answer is ``yes'' for both part \eqref{item:sup} and part \eqref{item:inf}.

Now we can already give the counterexample for Question~\ref{question:approxmatrixversion}\,\eqref{item:inf}.  Consider \(A = (p(z))\) with the polynomial \(p(z) = 5z^2 -6z + 5\).  The roots of \(p(z)\) are given by \(a = \frac{3}{5} + \frac{4}{5} \ima\) and its complex conjugate.  Let \(l \in (0,1)\) be determined by \(a = e^{2 \pi \ima l}\).  Since \(a\) is not a root of unity, the number \(l\) is irrational.  Let \(K\) be a positive integer.  Then Theorem~\ref{thm:Dirichlet} provides us with a sequence of positive integers \((i_j)\) such that \(0 < \|i_j K l\| \le \frac{1}{i_j}\).  This implies that we can find a \(K i_j\)-th root of unity \(\xi_{K i_j}\) with \(0 < |\xi_{K i_j} - a| \le 2 \sin(\frac{\pi}{K i_j^2}) \le \frac{2 \pi}{K i_j^2}\).  For sufficiently large \(j\) we obtain
\[ \sigma^+(A_{K i_j}) \le |p(\xi_{K i_j})| = 5 |\xi_{K i_j} - \overline{a}| |\xi_{K i_j} - a| \le 5 \cdot 2 \cdot \frac{2 \pi}{K i_j^2} \]
which gives
\[ \alpha(A_{K i_j}) \le \frac{\log\left(\frac{2}{K i_j}\right)}{\log\left(\frac{20 \pi}{K i_j^2}\right)} \]
hence \(\inf_{K \mid i} \alpha(A_i) \le \frac{1}{2}\).  Thus \(\liminf_{i \in F} \alpha(A_i) = \sup_{K \in F} \inf_{K | i} \alpha(A_i) \le \frac{1}{2}\) whereas \(\alpha^{(2)}(A) = 1\).

\subsection{General Laurent polynomials}  \label{subsection:generallaurentpolynomials}  Still let \(G = \Z\) but now let \(A = (p(z))\) for a general Laurent polynomial
\[ p(z) = c z^k \prod_{r=1}^s (z-a_r)^{\mu_r} \]
with \(c \in \C\), \(k \in \Z\) and the distinct roots \(a_r \in \C^*\) of \(p(z)\) of multiplicities \(\mu_r\).  We rearrange the roots of \(p(z)\) so that \(a_1, \ldots, a_u \in S^1\) and \(a_{u+1}, \ldots, a_s \notin S^1\) for some \(0 \le u \le s\).  By \cite{Lueck:L2Invariants}*{Lemma~2.58, p.\,100} and its proof we have that \(\alpha^{(2)}(p(z)) = \frac{1}{\max\{\mu_1, \ldots, \mu_u\}}\) if \(u \ge 1\) and \(\alpha^{(2)}(p(z)) = \infty^+\) otherwise.

To compute the alpha number of \(A_i\) in the case \(u \ge 1\), note that the singular values of \(A_i \in M(i,i; \C)\) are given by \(|p(\zeta_i^k)| = |c| \prod_{r=1}^s |\zeta_i^k - a_r|^{\mu_r}\) for \(k = 0, \ldots, i-1\).  Let \(d > 0\) and \(D > 0\) be given by the minimum and the maximum, respectively, of \(\prod_{r=u+1}^s |z-a_r|^{\mu_r}\) for \(z \in S^1\).  Let \(r_0 \le u\) be an index such that \(a_{r_0}\) is a root on the unit circle of maximal multiplicity \(\mu_0 = \mu_{r_0}\).  If \(a_{r_0}\) is an \(i\)-th root of unity, we have \(|\zeta_i^k - a_{r_0}| = 2 \sin(\frac{\pi}{i})\) where \(\zeta_i^k\) is either of the two \(i\)-th roots of unity adjacent to \(a_{r_0}\).  If \(a_{r_0}\) is not an \(i\)-th root of unity, then it lies in the open circle segment above one particular edge of the regular \(i\)-gon so that \(|\zeta_i^k - a_{r_0}| < 2 \sin(\frac{\pi}{i})\) for either of the two roots of unity \(\zeta_i^k\) spanning the segment.  In any case, we obtain that there exists \(0 \le k \le i-1\) with
\[ |p(\zeta_i^k)| \le |c| D 2^\mu \left(\sin\left(\frac{\pi}{i}\right)\right)^{\mu_0} \]
where \(\mu = \mu_1 + \cdots + \mu_u\).  Let us merge the constants to  \(K = |c| D 2^\mu\).  Since \(\sigma^+(A_i) \le |p(\zeta_i^k)|\) and \(m^+(A_i) \ge 1\), we have
\[  \alpha(A_i) \le  \frac{\log\left(\frac{1}{i}\right)}{\log\left(K \left(\sin\left(\frac{\pi}{i}\right)\right)^{\mu_0}\right)} = \frac{\log\left(\frac{1}{i}\right)}{\mu_0 \log\left(K^{\frac{1}{\mu_0}} \sin\left(\frac{\pi}{i}\right)\right)}. \]
A computation similar to the one in Section~\ref{section:preliminaries} gives \(\limsup_{i \rightarrow \infty} \alpha(A_i) \le \frac{1}{\mu_0}\), thus also \(\limsup_{i \in F}  \alpha(A_i) \le \frac{1}{\mu_0}\) by Lemma~\ref{lemma:sequencenet}.  To show equality (in both cases) it remains to identify \(\frac{1}{\mu_0}\) as a cluster point of the net \((\alpha(A_i))_{i \in F}\).  This is the tricky part.

Note that the notation \(\|x\|\) from Section~\ref{section:preliminaries} still makes sense and is well-defined for \(x \in \mathbb{R} / \mathbb{Z} = \mathbb{T}\).  The same two inequalities from before hold true and even better, the term \(\|x - y\|\) for \(x, y \in \mathbb{T}\) defines a metric inducing the given topology on \(\mathbb{T}\).

\begin{proposition} \label{prop:middlengonedge}
For all points \(z_1, \ldots, z_u  \in S^1 \subset \C\) on the circle there is \(0<R<\frac{1}{2}\) such that for each positive integer \(K\) there are infinitely many positive integers \(i_j\) such that for all \(t = 1, \ldots, u\) and for all \( k = 1, \ldots, K i_j\) either
\[ z_t = \zeta^k_{K i_j} \quad \text{or} \quad |z_t-\zeta^k_{K i_j}| \ge 2 \sin \left(\frac{R\pi}{K i_j}\right) \]
where \(\zeta_{K i_j}\) is a fixed primitive \(K i_j\)-th root of unity.
\end{proposition}

\begin{proof}
For what comes next it is preferable to think of the \(u\)-dimensional torus as the additive group \(\mathbb{T}^u = (\mathbb{R} / \mathbb{Z})^u\).  Accordingly, let us change the notation for the point \((z_1, \ldots, z_u)\) in \((S^1)^u\) to \(L = (L_1, \ldots, L_u)\) in \(\mathbb{T}^u\) so that \((z_1^n, \ldots, z_u^n)\) corresponds to \(nL = (nL_1, \ldots, nL_u)\).  The point \(L\) defines a homomorphism of \(\Z\)-modules (abelian groups) \(\varphi_L \colon \Z^u \rightarrow \mathbb{T} = \R/\Z\) sending \((a_1, \ldots, a_u) \in \Z^u\) to \(\sum_{j=1}^u a_j L_j \in \mathbb{T}\).  Let \(\{A_1, \ldots, A_k \} \subset \Z^u\) be a basis of the free submodule \(\ker \varphi_L\) of \(\Z^u\).  Considering these basis elements as the columns of a \((u \times k)\)-matrix \(A\), they define a homomorphism \(\R^u \rightarrow \R^k\) where we write elements of \(\R^u\) and \(\R^k\) as row vectos and multiply them from the right with \(A\).  This homomorphism descends to a homomorphism \(\psi_A \colon \mathbb{T}^u \rightarrow \mathbb{T}^k\).  Theorem~\ref{thm:Kronecker} says precisely that the \(\Z\)-orbit \( B_L = \{ nL \in \mathbb{T}^u \, | \, n \in \Z \} \) of \(L\) in the \(u\)-torus \(\mathbb{T}^u\) has closure \(\overline{B_L} = \ker(\psi_A)\).  It follows from this description that \(\overline{B_L} \cong \mathbb{T}^v \oplus \Z / m \Z\) for some \(m \ge 1\), compare also \cite{Abbaspour-Moskowitz:Basic}*{Corollary~4.2.5, p.\,209}.  Here the dimension \(v\) is one less than the dimension of the \(\Q\)-vector space generated by \(1, \widetilde{L_1}, \ldots, \widetilde{L_u}\) where each \(\widetilde{L_t}\) is some lift of \(L_t\) from \(\mathbb{T}\) to \(\R\).  Therefore \(v\), depending on \(L\), can take any value between zero and \(u\).  For the moment, let us assume \(v \ge 1\).  Since the quotient \(\overline{B_L} / \overline{B_L}^0 \cong \Z / m \Z\) by the unit component is generated by \(L + \overline{B_L}^0\), it follows that \(\overline{B_{mL}} = \overline{B_L}^0 \cong \mathbb{T}^v\).  Let
\[ \mathbb{T}_{mL} = \{ (x_1, \ldots, x_u) \in \mathbb{T}^u \,|\, x_t = [0] \text{ if } mL_t = [0] \} \]
be the unique minimal subtorus obtained from \(\mathbb{T}^u\) by setting fixed coordinates to zero under the side condition that it still contains \(\overline{B_L}^0\).  It is then of course necessary that \(1 \le v \ \le \ l = \dim \mathbb{T}_{mL} \ \le \ u\).  In what follows we will delete the zero coordinates from \(\mathbb{T}_{mL}\).
We can choose \(0 < R < \frac{1}{2}\) so small that the interior of the centered cube
\[  K_R = \{ (x_1, \ldots, x_l) \in \mathbb{T}_{mL} \,|\, {\textstyle \|x_t\| \ge R \text{ for all } t = 1, \ldots, l} \}. \]
contains \([mL]\) and therefore intersects \(\overline{B_L}^0\) in the nonempty set \(U_L\).  Next we claim that for every nonzero \(K \in \Z\) we have \(\overline{B_{KmL}} = \overline{B_{mL}} = \overline{B_L}^0\).

Indeed, the inclusion \(\overline{B_{KmL}} \subset \overline{B_{mL}}\) is clear.  For the other inclusion we note that \(\overline{B_{mL}} = \overline{B_L}^0 \cong \mathbb{T}^v\) is a torus, hence is divisible.  Thus for given \(x \in \overline{B_{mL}}\) and \(\varepsilon > 0\) there is \(y =(y_1, \ldots, y_u) \in \overline{B_{mL}}\) such that \(Ky = x\) and there is \(N \in \Z\) such that \(\|N m L_t - y_t\| < \frac{\varepsilon}{\lvert K \rvert}\) for all \(t = 1, \ldots, u\).  It follows that \(\|N (KmL_t) - x_t\| = \|K(NmL_t - y_t)\| < \varepsilon\), hence \(x \in \overline{B_{KmL}}\).

Since \(U_L\) is open in \(\overline{B_L}^0\), it contains infinitely many \(\Z\)-translates of \(K m L\).  Note moreover that \(U_L = -U_L\), so we can pick a sequence \(i_j\) of positive integer multiples of \(m\) such that \(i_j K L \in U_L\) for all \(j\).

By construction we have that for each \(i_j\) either \(K i_j L_t = [0]\), meaning that \(z_t \in S^1\) is a \(K i_j\)-th root of unity, or \(\|K i_j L_t \| \ge R\), meaning that \(z_t\) encloses an angle of at least \(\frac{2 \pi R}{K i_j}\) with any \(K i_j\)-th root of unity.  This gives the assertion for \(v \ge 1\).  In case \(v = 0\) we have \(L_t \in \Q / \Z\) for all \(t = 1, \ldots, u\) or in other words each \(z_t \in S^1\) is some \(k_t\)-th root of unity.  In that case setting \(i_j = j \lcm (k_1, \ldots, k_u) \) does the trick for arbitrary \(0 < R < \frac{1}{2}\).
\end{proof}

\noindent To see that \(\frac{1}{\mu_0}\) is a cluster point of the net \((\alpha(A_i))_{i \in F}\), for any given positive integer \(K\) we have to construct a sequence of positive integers \(i_j\) such that \(\lim_{j \rightarrow \infty} \alpha(A_{K{i_j}}) = \frac{1}{\mu_0}\).  So let the number \(0 < R < \frac{1}{2}\) and the sequence \((i_j)\) be specified by \(a_1, \ldots, a_u \in S^1\) and by \(K\) according to Proposition~\ref{prop:middlengonedge}.  We now ask for a lower bound on \(\sigma^+(A_{K i_j})\).  Let \(\delta = \min \{ \frac{1}{2}, \eta \}\) where \(\eta\) is the minimum of the pairwise Euclidean distances of the points \(\{a_1, \ldots, a_u\} \subset S^1\).  Let \(\xi_{K i_j}\) be (one of) the \(K i_j\)-th root(s) of unity for which \(\sigma^+(A_{K i_j}) = \lvert p(\xi_{K i_j}) \rvert\).  For sufficiently large \(j\), there must be one and only one root \(a_r\) on \(S^1\) within the open \(\delta\)-ball around \(\xi_{K i_j}\), where \(r = r(j)\) depends on \(j\).    So if \(a_{r(j)}\) is not a \(K i_j\)-th root of unity, we have \(\lvert \xi_{K i_j} - a_{r(j)} \rvert \ge 2 \sin \frac{R \pi}{K i_j}\) and if \(a_{r(j)}\) is a \(K i_j\)-th root of unity, we even have \(\lvert \xi_{K i_j} - a_{r(j)} \rvert \ge 2 \sin \frac{\pi}{K i_j}\).  For sufficiently large \(j\), this gives
\begin{align*}
|p(\xi_{K i_j})| \ge |c| d \delta^{\mu -\mu_{r(j)}} \left(2\sin\left(\frac{R\pi}{K i_j}\right)\right)^{\mu_{r(j)}} \ge |c| d \delta^\mu 2^{\mu_0} \left(\sin\left(\frac{R\pi}{K i_j}\right)\right)^{\mu_0}.
\end{align*}
Since \(p\) is a polynomial, the function \(t \mapsto \lvert p(e^{2\pi \ima t})\rvert\) is strictly monotonic on small half-open intervals starting at the zeros and the function is bounded from below outside these intervals.  Thus for large \(j\) we  have \(m^+(A_{K i_j}) \le 2u\).  (Note that we use the symbol ``\(\ima\)'' for the imaginary unit whereas the symbol ``\(i\)'' is reserved for indices.)  The same computation as above shows \(\liminf_{j\rightarrow \infty} \alpha(A_{K i_j}) \ge \frac{1}{\mu_0}\), thus \(\lim_{j\rightarrow \infty} \alpha(A_{K i_j}) = \frac{1}{\mu_0}\).  This answers Question~\ref{question:approxmatrixversion}\,\eqref{item:sup} affirmatively for the case \(G = \Z\) and \(r = s = 1\).

\section{The case of a matrix of Laurent polynomials}
\label{section:rbys}

\noindent For a general matrix \(A \in M(r,s; \C[\Z])\) with arbitrary \(r, s\) we notice that the ring of Laurent polynomials \(\C[\Z]\), being a localization of the polynomial ring \(\C[z]\), is a principal ideal domain.  Therefore \(A\) can be transformed into \emph{Smith normal form}.  This means there are invertible matrices \(S \in M(r,r; \C[\Z])\) and \(T \in M(s,s; \C[\Z])\) such that \(SAT\) is an \((r \times s)\)-matrix of block form \(\left( \begin{smallmatrix} P & 0 \\ 0 & 0 \end{smallmatrix} \right)\) where \(P\) is a diagonal matrix with entries \(p_1(z), \ldots, p_k(z)\).

The (Laurent) polynomials \(p_1(z), \ldots, p_k(z)\) are called the \emph{invariant factors} and satisfy the relation \(p_l \mid p_{l+1}\).  Multiplying \(S\) or \(T\) by a diagonal matrix with nonzero constant polynomials as entries, if need be, we can and will additionally assume that \(|p_{l+1}(z)| \le |p_l(z)|\) for all \(z \in S^1\) and \(l=1, \ldots, k-1\).  By \cite{Lueck:L2Invariants}*{Lemma~2.11\,(9), p.\,77, and Lemma~2.15\,(1), p.\,80} we get
\[ \alpha^{(2)}(A) = \alpha^{(2)}(SAT) = \min_{l=1,\ldots, k} \{ \alpha^{(2)}(p_l(z)) \} = \alpha^{(2)}(p_k(z)). \]
The last equality holds because the maximal multiplicity of a root on the unit circle can only increase from \(p_l\) to \(p_{l+1}\).  The following proposition thus reduces Question~\ref{question:approxmatrixversion} for the \((r \times s)\)-matrix \(A\) to the same question for the \((1 \times 1)\)-matrix \((p_k(z))\).  The latter was treated in the preceding section.

\begin{proposition} \label{prop:rbystoonebyone} Suppose \(\alpha^{(2)}(A) < \infty^+\).  Then we have
  \[ \liminf_{i \in F} \alpha(A_i) = \liminf_{i \in F} \alpha(p_k(z)_i) \ \ \text{and} \ \ \limsup_{i \in F} \alpha(A_i) = \limsup_{i \in F} \alpha(p_k(z)_i) \]
and the same statement holds replacing ``\(i \in F\)'' with ``\(i \rightarrow \infty\)''.
\end{proposition}

\noindent The proof requires some labor.  We prepare it with a lemma that captures those properties of the functions \(t \mapsto \lvert p_l(e^{\ima 2 \pi t}) \rvert\) that are relevant for computing alpha numbers.

\begin{lemma} \label{lemma:growthnearroots}
  Let \(p_1(z), \ldots, p_k(z) \in \C[z,z^{-1}]\) be complex Laurent polynomials.  Then there exists \(0 < \varepsilon < 1\) and there exist constants \(d, D > 0\) such that for every polynomial \(p_l(z)\)

  \begin{enumerate}[(i)]
  \item \label{item:estimatearoundroots} we have the inequality
    \[ d|t|^{\mu} \le |p_l(a e^{\ima 2 \pi t})| \le D|t|^{\mu} \]
    for every root \(a\) of \(p_l(z)\) on \(S^1\) and each \(t \in (-\varepsilon, \varepsilon)\) where \(\mu\) is the multiplicity of \(a\),
    \item \label{item:monotonearoundroots} the function \(|p_l(ae^{\ima 2\pi t})|\) is monotone decreasing for \(t \in (-\varepsilon, 0]\) and monotone increasing for \(t \in [0, \varepsilon)\) for every root \(a\) of \(p_l(z)\) on \(S^1\),
    \item \label{item:estimateoutsideroots} the function \(|p_l(e^{\ima 2\pi t})|\) is bounded from below by \(d \varepsilon^{\mu_0}\) on the complement of all open \(\varepsilon\)-balls around the roots of \(p_l(z)\) on \(S^1\) where \(\mu_0\) is the maximal multiplicity among all the roots of all polynomials \(p_1(z), \ldots, p_k(z)\).
      \end{enumerate}
\end{lemma}

\begin{proof}
 Let \(a \in S^1\) be a root of \(p_l(z)\) of multiplicity \(\mu\).  Let \(0 < \delta < 2\) be so small that \(p(z)\) has no second root in \(B_\delta(a)\),  the closed \(\delta\)-ball around \(a\).  Let \(d' > 0\) and \(D' > 0\) be given by the minimum and maximum, respectively, of \(\left| \frac{p(z)}{(z-a)^\mu}\right|\) for \(z \in B_\delta(a)\).  Set \(\varepsilon = \frac{1}{\pi} \arcsin(\frac{\delta}{2})\), so that in particular \(\varepsilon\) is bounded from above by \(\frac{1}{2}\), and set \(d = d' 4^\mu\) and \(D = D' (2 \pi)^\mu\). Then for \(|t| < \varepsilon\) we have
  \begin{align*}
|p_l(a e^{\ima 2 \pi t})| & \le D' |ae^{\ima 2 \pi t} - a|^\mu = D' |e^{\ima 2 \pi t} - 1|^\mu = D' 2^\mu |\sin(\pi t)|^\mu \\
& \le D' (2\pi)^\mu|t|^\mu = D |t|^\mu
\end{align*}
and similarly
\[ |p_l(a e^{\ima 2 \pi t})| \ge d' 2^\mu |\sin(\pi t)|^\mu \ge d' 4^\mu |t|^\mu = d |t|^\mu. \]
We repeat this construction for all the remaining roots of \(p_l(z)\) on \(S^1\) and for all the remaining polynomials.  The minimal occurring \(\varepsilon\) and \(d\) together with the maximal occurring \(D\) will then work for all roots and polynomials and gives \eqref{item:estimatearoundroots}.  It is clear that since \(p_l(z)\) is a polynomial, we can additionally achieve \eqref{item:monotonearoundroots} and \eqref{item:estimateoutsideroots} by making \(\varepsilon\) smaller, if necessary.
\end{proof}

\begin{proof}[Proof of Proposition~\ref{prop:rbystoonebyone}.]
For any two matrices \(M, N \in M(n,n; \C)\) we have the inequalities of singular values for each \(t = 1, \ldots, n\)
\begin{align}
\label{eq:ineqsingvalues1} \sigma_n(M) \sigma_t(N) & \le \sigma_t(MN) \le \sigma_1(M) \sigma_t(N) \\
\label{eq:ineqsingvalues2} \sigma_t(M) \sigma_n(N) & \le \sigma_t(MN) \le \sigma_t(M) \sigma_1(N)
\end{align}
as given for instance in \cite{Hogben:HandbookLinAl}*{24.4.7\,(c), p.\,24-8}.  Here, as usual, the singular values are listed in nonincreasing order.  Of course the second inequality follows from the first because \(\sigma_t(M) = \sigma_t(M^\top)\).  We apply these inequalities to our setting as follows.  Let \(m  = \max \{r, s\}\) and view the matrices \(A_i\) as lying in \(M(mi,mi; \C)\) by embedding \(A_i\) in the upper left corner of an \((mi \times mi)\)-matrix, filling up the remaining entries with zeros.  If \(r < s\) we consider \(S_i\) as an element of \(\textup{GL}(mi; \C)\) by overwriting the upper left block of an \((mi \times mi)\)-identity matrix with \(S_i\) and similarly for \(T_i\) in place of \(S_i\) if \(r > s\).  Since both \(S\) and \(T\) are invertible over the group ring \(\C[\Z]\), it follows from \cite{Lueck:L2Invariants}*{Lemma~13.33, p.\,466} that the spectrum of \(r^{(2)}_{SS^*}\) and \(r^{(2)}_{TT^*}\) is contained in \([C^{-1}, C]\) for some \(C \ge 1\).  Since the operator norm of the projection map \(L^1(G) \rightarrow L^1(G/G_i)\) is bounded by one, it follows that the eigenvalues of \((SS^*)_i\) and \((TT^*)_i\) are likewise constrained to lie within \([C^{-1}, C]\).  Therefore \(C^{-\frac{1}{2}} \le \sigma_t(S_i), \sigma_t(T_i) \le C^{\frac{1}{2}}\) for each \(t = 1, \ldots, mi\) so that the inequalities~\eqref{eq:ineqsingvalues1} and~\eqref{eq:ineqsingvalues2} give
\begin{equation} \label{eq:boundedsingvalue} C^{-1} \sigma_t((SAT)_i)  \le \sigma_t(A_i) \le C \sigma_t((SAT)_i). \end{equation}
The special case \(t = \rank_\C (A_i)\) gives
\begin{equation} \label{eq:boundedsmallestsingvalue} C^{-1} \sigma^+((SAT)_i)  \le \sigma^+(A_i) \le C \sigma^+((SAT)_i). \end{equation}
Next we show that there is \(M > 0\) such that
\begin{equation} \label{eq:boundedmultiplicities} 1 \le m^+(A_{i}) \le M \end{equation}
for all sufficiently large \(i\).  To this end, let \(r_i = \rank_\C A_i\) so that \(\sigma^+(A_i) = \sigma_{r_i}(A_i)\).  Let \(\mu_0\) be the maximal occurring multiplicity among the roots of \(p_k(z)\) on \(S^1\).  Let \(\varepsilon > 0\) and \(d, D > 0\) be the constants from Lemma~\ref{lemma:growthnearroots} applied to the polynomials \(p_1(z), \ldots, p_k(z)\) which form the diagonal of the matrix \(SAT\).  Pick a positive integer
\begin{equation} \label{eq:constantk} K > \left(\frac{C^2 D}{d}\right)^{\frac{1}{\mu_0}} + 1 \end{equation}
and set \(\delta = \frac{d}{D}  \varepsilon^{\mu_0}\).  Now we consider \(i\) so large that at least \(2K\) of the \(i\)-th roots of unity lie in any open \(\delta\)-ball around any point on \(S^1\).  By Lemma~\ref{lemma:growthnearroots}\,\eqref{item:estimatearoundroots} and \eqref{item:monotonearoundroots}, evaluating the function \(|p_l(z)|\) in the \(2K\) roots of unity closest to any root \(a \in S^1\) gives values smaller than \(D (\frac{d}{D} \varepsilon^{\mu_0})^\mu \le d \varepsilon^{\mu_0}\) where \(\mu\) was the multiplicity of \(a\).  So if \(N\) denotes the sum of the number of distinct roots of each \(p_l(z)\), then by Lemma~\ref{lemma:growthnearroots}\,\eqref{item:estimateoutsideroots} the first \(2KN\) (positive) singular values of \((SAT)_i\) are given by evaluating some \(|p_l(z)|\) within the \(\varepsilon\)-ball of some root.  By the pigeon hole principle there is one root \(a \in S^1\) of some \(p_l(z)\) such that \(K\) singular values among the smallest \(2KN\) singular values of \((SAT)_i\) are given by evaluating \(|p_l(z)|\) at the \(K\) closest \(i\)-th roots of unity on one side of the root \(a\). Again we denote the multiplicity of \(a\) by \(\mu\).  Using the monotonicity asserted by Lemma~\ref{lemma:growthnearroots}\,\eqref{item:monotonearoundroots} this gives 
\[ \sigma_{{r_i} - 2NK} ((SAT)_i) \ge \left|p_l\left(a e^{\pm \ima \frac{2 \pi}{i} (K - 1)}\right)\right|. \]
Applying Lemma~\ref{lemma:growthnearroots}\,\eqref{item:estimatearoundroots} and inequality~\eqref{eq:constantk} we get
\begin{align*}
\left|p_l\left(a e^{\pm \ima \frac{2 \pi}{i} (K - 1)}\right)\right| \ge d\left(\frac{K-1}{i}\right)^\mu \ge d\left(\frac{K-1}{i}\right)^{\mu_0} > C^2 D \left(\frac{1}{i}\right)^{\mu_0}
\end{align*}
Let \(a_0 \in S^1\) be any root of \(p_k(z)\) with multiplicity \(\mu_0\).  There is an \(i\)-th root of unity \(\xi_i \ne a_0\) which encloses an angle of at most \(\frac{2\pi}{i}\) with \(a_0\).  Applying Lemma~\ref{lemma:growthnearroots}\,\eqref{item:estimatearoundroots} again we obtain
\[ C^2 D \left(\frac{1}{i}\right)^{\mu_0} \ge C^2 \left|p_k\left(a_0 e^{\pm \ima \frac{2 \pi}{i}}\right)\right| \ge C^2 |p_k(\xi_i)| \ge C^2 \sigma^+((SAT)_i). \]
So setting \(M = 2NK\) we have \(\sigma_{r_i - M} ((SAT)_i) > C^2 \sigma^+((SAT)_i)\) for every large enough \(i\).  From inequality~\eqref{eq:boundedsingvalue} we conclude
\[ \sigma_{r_i - M}(A_i) \ge C^{-1} \sigma_{r_i - M}((SAT)_i) > C \sigma^+((SAT)_i) \ge \sigma^+(A_i) \]
which proves inequality~\eqref{eq:boundedmultiplicities}.

Finally note that the inequality \(|p_{l+1}(z)| \le |p_l(z)|\) gives \(\sigma^+((SAT)_i) = \sigma^+((p_k(z))_i)\). Inequalities~\eqref{eq:boundedsmallestsingvalue} thus yields
\begin{equation} \label{eq:boundonalphanumbers}
  \frac{\log\left(\frac{m^+(A_i)}{i}\right)}{\log (C^{-1}\sigma^+((p_k(z))_i))} \le \alpha(A_i) \le \frac{\log\left(\frac{m^+(A_i)}{i}\right)}{\log (C\sigma^+((p_k(z))_i))}.
\end{equation}
We can rewrite the outer terms as
\[ \frac{\log\left(\frac{m^+(A_i)}{i}\right)}{\log (C^{\pm 1}\sigma^+((p_k(z))_i))} = \frac{\log\left(\frac{m^+(p_k(z)_i)}{i}\right) + \log\left(\frac{m^+(A_i)}{m^+(p_k(z)_i)}\right)}{\log\left(\sigma^+(p_k(z)_i)\right)\left(1 \pm \frac{\log C}{\log(\sigma^+(p_k(z)_i))}\right)}. \]
Since the multiplicities are bounded according to inequality~\eqref{eq:boundedmultiplicities}, we see from this that for an increasing sequence of positive integers \((i_j)\) we have \(\lim_{j \rightarrow \infty} \alpha(A_{i_j}) = c\) if and only if \(\lim_{j \rightarrow \infty} \alpha(p_k(z)_{i_j}) = c\).  As a consequence the sequences \((\alpha(A_i))_{i \ge 0}\) and \((\alpha(p_k(z)_i))_{i \ge 0}\) share the same set of cluster points.  Considering integer sequences of the form \((K i_j)\) for any positive integer \(K\), the same goes for the nets \((\alpha(A_i))_{i \in F}\) and \((\alpha(p_k(z)))_{i \in F}\).  This clearly implies the proposition.
\end{proof}

\noindent This answers Question~\ref{question:approxmatrixversion}\,\eqref{item:sup} affirmatively for the case \(G = \Z\).

\section{The case of a virtually cyclic group}
\label{section:virtuallycyclic}

\noindent Finally let \(G\) be infinite virtually cyclic so that \(G\) contains an infinite cyclic subgroup \(Z \le G\) with \([G \colon Z] = n < \infty\).  By going over to the normal core, if need be, we can and will assume that \(Z\) is a normal subgroup.
We choose representatives \(g_i \in G\) such that \(Z \backslash G = \{ Z g_1, \ldots, Z g_n\}\).  Let \(A \in M(r,s; \C G)\).  Right multiplication with \(A\) defines a homomorphism \((\C G)^r \rightarrow (\C G)^s\) of left \(\C G\)-modules.  If we consider \(\C G\), the free left \(\C G\)-module of rank one, as a left \(\C Z\)-module, then it is free of rank \(n\) and a basis is given by \(g_1, \ldots, g_n \in \C G\).  Accordingly, viewing right multiplication with \(A\) as a homomorphism \((\C Z)^{rn} \rightarrow (\C Z)^{sn}\) of left \(\C Z\)-modules, it is given by right multiplication with the matrix \(\res^Z_G(A) \in M(rn, sn; \C Z)\) that results from \(A\) by replacing the \((p,q)\)-th entry \(\sum_{g \in G} \lambda^{p,q}_g g\) with the \((n \times n)\)-matrix over \(\C Z\) whose \((u,v)\)-th entry is \(\sum_{h \in Z} \lambda^{p,q}_{g_u^{-1} h g_v} h\) for \(1 \le u, v \le n\).

Let \(Z_i\) be the unique subgroup of \(Z\) with \([Z : Z_i] = i\).  Then \([G : Z_i] = ni\) and \(Z_i\) is normal in \(G\) because \(Z_i\) is characteristic in \(Z\).

\begin{proposition} \label{prop:equalmatrices}
We have \(\res^Z_G(A)_i = A_i\) as elements in \(M(r n i, s n i; \C)\).
\end{proposition}

\begin{proof}
We pick representatives \(Z_i \backslash Z = \{ Z_i h_1, \ldots, Z_i h_i \}\) and verify that for \(1 \le p \le r\) and \(1 \le q \le s\) as well as \(1 \le u, v \le n\) we have
\[ (\res^Z_G(A)_i)_{(p-1)n + u,(q-1)n + v} = \sum_{l=1}^i \left( \sum_{h \in Z_i} \lambda_{g_u^{-1} h h_l g_v}^{p,q} \right) Z_i h_l. \]
Multiplication with a fixed coset \(Z_i h_k\) gives
\[ Z_i h_k \sum_{l=1}^i \left( \sum_{h \in Z_i} \lambda_{g_u^{-1} h h_l g_v}^{p,q} \right) Z_i h_l = \sum_{l=1}^i \left( \sum_{h \in Z_i} \lambda_{g_u^{-1} h_k^{-1} h h_l g_v}^{p,q} \right) Z_i h_l. \]
Hence \(\res_G^Z(A)_i\) is realized over \(\C\) by replacing the entry at \(((p-1)n + u,(q-1)n + v)\) with a (circulant) \((i \times i)\)-matrix whose \((k,l)\)-th entry is \(\sum_{h \in Z_i} \lambda_{g_u^{-1} h_k^{-1} h h_l g_v}^{p,q}\).

To realize \(A_i\) as a matrix over \(\C\) we now use our chosen representatives to list the cosets of \(Z_i \backslash G\) in this order as
\[\{ Z_i h_1 g_1, \ldots,   Z_i h_i g_1, \quad \ldots \quad ,  Z_i h_1 g_n, \ldots,  Z_i h_i g_n\}. \]
Again we compute for \(1 \le p \le r\) and \(1 \le q \le s\) as well as \(1 \le u, v \le n\) and \(1 \le k \le i\) that
\[ Z_i h_k g_u \sum_{g \in G} \lambda^{p,q}_g Z_i g = \sum_{g \in G} \lambda^{p,q}_{(h_k g_u)^{-1} g} Z_i g = \sum_{v = 1}^n \sum_{l = 1}^i \sum_{h \in Z_i} \lambda^{p,q}_{g_u^{-1} h_k^{-1} h h_l  g_v} Z_i h_l g_v. \]
Thus \(A_i\) is realized over \(\C\) by replacing the \((p,q)\)-th entry with the \((n i \times n i)\)-matrix whose entry at \(((u-1)i + k, (v-1)i + l)\) is \(\sum_{h \in Z_i} \lambda^{p,q}_{g_u^{-1} h_k^{-1} h h_l  g_v}\).  Thus the \(\C\)-matrices \(\res_G^Z(A)_i\) and \(A_i\) coincide.
\end{proof}

\begin{proposition} \label{prop:equallimits}
Let \(F(Z)\) and \(F(G)\) denote the full residual systems of \(Z\) and \(G\), respectively.  Suppose \(\alpha^{(2)}(A) < \infty^+\), then
\[ \liminf_{i \in F(Z)} \alpha(A_i) = \liminf_{i \in F(G)} \alpha(A_i) \quad \text{and} \quad \limsup_{i \in F(Z)} \alpha(A_i) = \limsup_{i \in F(G)} \alpha(A_i). \]
\end{proposition}

\begin{proof}
Let \(c\) be a cluster point of the net \((\alpha(A_i))_{i \in F(Z)}\) and let \(H \trianglelefteq G\) be a finite index normal subgroup representing some element in \(F(G)\).  Then there are upper bounds \(j\) of \(H \cap Z\) in \(F(Z) \subset F(G)\) with \(\alpha(A_j)\) arbitrarily close to \(c\).  Conversely, let \(c\) be a given cluster point of the net \((\alpha(A_i))_{i \in F(G)}\) and consider \(Z_i \trianglelefteq Z\).  Then \(Z_i\) represents an element in \(F(G)\), thus there are upper bounds \(j\) of \(Z_i\) in \(F(G)\), which actually lie in \(F(Z)\), with \(\alpha(A_j)\) arbitrarily close to \(c\).  Thus the set of cluster points agrees for the nets \((\alpha(A_i))_{i \in F(Z)}\) and \((\alpha(A_i))_{i \in F(G)}\) which in particular implies the proposition.
\end{proof}

\noindent Now we are in the position to complete the proof of our main result.

\begin{proof}[Proof of Theorem~\ref{thm:approxvcycmatrixversion}]  It follows from \cite{Lueck:L2Invariants}*{Theorem~1.12\,(6), p.\,22} that for the spectral distribution functions we have \(F_{\res^Z_G(A)}(\lambda) = n F_A(\lambda)\), hence \(\alpha^{(2)}(A) = \alpha^{(2)}(\res^Z_G(A))\).  Together with the preceding section, Proposition~\ref{prop:equalmatrices} and Proposition~\ref{prop:equallimits} we obtain
\begin{align*}
\alpha^{(2)}(A) &= \alpha^{(2)}(\res^Z_G(A)) = \limsup_{i \in F(Z)} \alpha(\res^Z_G(A)_i) = \\
&= \limsup_{i \in F(Z)} \alpha(A_i) = \limsup_{i \in F(G)} \alpha(A_i).
\end{align*}
This answers Question~\ref{question:approxmatrixversion}\,\eqref{item:sup} in the affirmative for \(F = \C\) and thus for any subfield.  In Section~\ref{subsection:twoexamples} we gave an example answering Question~\ref{question:approxmatrixversion}\,\eqref{item:inf} in the negative for \(F = \Q\) and thus for every larger field.
\end{proof}

\section{The lower limit of alpha numbers} \label{section:liminfpositive}

\noindent In this final section we give the proof of Theorem~\ref{thm:liminfpositive}.  Recall our definition of Baker constants from the end of Section~\ref{subsection:diophantine}.

\begin{theorem} \label{thm:circlerunner}
  Let \(a \neq 1\) be an algebraic number on the unit circle and let \(D\) be a Baker constant of the pair \((a, -1)\).  Then for all \(n \ge 2\) with \(a^n \neq 1\) we have \(\lvert a^n - 1 \rvert \ge \frac{n^{-D}}{2}  \).
\end{theorem}

\begin{proof}
The principal value logarithm satisfies \(\lvert \log (1 + z) \rvert \le 2 \lvert z \rvert\) for \(\lvert z \rvert \le \frac{1}{2}\) and is additive up to some integer multiple of \(2 \pi \ima\).  If \(\lvert a^n - 1 \rvert > \frac{1}{2}\), there is nothing to prove.  Otherwise we have
\[ 1 \ge 2 \lvert a^n - 1 \rvert \ge \lvert \log(a^n) \rvert = \lvert n \log a + 2 \pi \ima k \rvert = \lvert n \log a + 2k \log(-1) \rvert, \]
so if \(a^n \ne 1\), Theorem~\ref{thm:littlebaker} gives \(2 \lvert a^n - 1 \rvert \ge \max \{ n, 2 \lvert k \rvert \}^{-D}\).  Moreover, the inequalities \(1 \ge \lvert n \log a + 2 \pi \ima k \rvert\) and \(\lvert \log a \rvert \le \pi\) imply
\[ \lvert k \rvert \le \frac{1 + n \lvert \log a \rvert}{2 \pi} \le \frac{1}{2 \pi} + \frac{n}{2} \]
which is equivalent to \(\lvert k \rvert \le \frac{n}{2}\) because \(k\) and \(n\) are integers.  Thus we obtain \(2 \lvert a^n - 1 \rvert \ge n^{-D} \) as desired.
\end{proof}

\noindent The mere existence of some \(D > 0\) giving the estimate of the theorem also serves as the main ingredient for \cite{Lueck:L2Invariants}*{Lemma~13.53, p.\,478}.  The latter is just the \((1 \times 1)\)-case of the Fuglede--Kadison determinant approximation conjecture for the group \(\Z\).  We recapped a proof here, however, in order to identify the constant \(D\) as the Baker constant in Theorem~\ref{thm:littlebaker}.  This has the virtue that the many estimates on \(D\) in the literature lead to explicit lower bounds on our \(\liminf_{i \in F} \alpha(A_i)\) as we will see in the subsequent corollary.  We admit that the practical value of these bounds is limited because the values for \(D\) given in the literature are typically astronomic.  The constant in \cite{Baker:LinearForms}*{Theorem~2}, for example, is \(D = (32d)^{400}\) times a logarithmic function in the height of \(a\), where \(d\) is the degree of \(a\).

\begin{corollary} \label{cor:boundonliminf}
  Let \(G\) be a virtually cyclic group and let \(A \in M(r,s;\Q G)\) with \(\alpha^{(2)}(A) < \infty^+\).  Choose an infinite cyclic normal subgroup \(Z \trianglelefteq G\) of finite index and let \(p_k(z)\) be the maximal invariant factor of \(\res^G_Z(A)\).  We denote the zeros of \(p_k(z)\) on \(S^1\) by \(a_1, \ldots, a_u\) and let \(D\) be the maximal occurring Baker constant \(D = D(a_t, -1)\) for \(a_t \neq 1\).  Then
\[ \liminf_{i \in F} \alpha(A_i) \ge \frac{\alpha^{(2)}(A)}{1 + D}. \]
\end{corollary}

\begin{proof}
  Again let \(\mu_0\) be the maximal multiplicity amongst the roots \(a_1, \ldots, a_u\) of the polynomial \(p_k(z)\) which lie on \(S^1\).  As explained in the previous two sections we have
  \[ \alpha^{(2)}(A) = \alpha^{(2)}(\res^Z_G(A)) = \alpha^{(2)}(p_k(z)) = \textstyle \frac{1}{\mu_0}. \]
Fix \(\varepsilon > 0\) and consider \(i \ge 2^{\frac{1}{\varepsilon}}\).  Let \(\zeta_i\) be a primitive \(i\)-th root of unity.  For every \(a_t\) which is not an \(i\)-th root of unity, Theorem~\ref{thm:circlerunner} gives us
\(\lvert a_t^i - 1 \rvert \ge \frac{1}{2} i^{-D} \ge i^{-(D + \varepsilon)}\) and therefore
\begin{equation} \label{eq:dplusone}
  \lvert a_t - \zeta_i^l \vert = \frac{\lvert a_t^i - 1 \rvert}{\left\lvert \sum_{j = 0}^{i-1} a_t^{i-j-1} \zeta_i^{lj} \right\rvert} \ge \frac{1}{i^{D+1 + \varepsilon}}
\end{equation}
for every \(l = 0, \ldots, i-1\).  Let \(\xi_i\) be the (or an) \(i\)-th root of unity for which \(\sigma^+(p_k(z)_i) = \lvert p_k(\xi_i) \rvert\).  Let \(c\), \(d\), \(\mu\) and \(\delta\) be the constants from below the proof of Proposition~\ref{prop:middlengonedge}.  As before, for large enough \(i\) there is one and only one root \(a_{r(i)}\) of \(p_k(z)\) with multiplicity \(\mu_{r(i)}\) that lies within the open \(\delta\)-ball around \(\xi_i\).  If \(a_{r(i)}\) is an \(i\)-th root of unity and \(i\) is large enough, then \(\xi_i\) must be one of the two \(i\)-th roots of unity adjacent to \(a_{r(i)}\) so that we get
\begin{equation} \label{eq:ifrootofunity} \lvert a_{r(i)} - \xi_i \rvert = 2 \sin \left( \frac{\pi}{i} \right) \ge \frac{1}{i}. \end{equation}
So in any case, either from equation~\eqref{eq:dplusone} or from equation~\eqref{eq:ifrootofunity}, we get
\[ \sigma^+(p_k(z)_i) = \lvert p_k(\xi_i) \rvert \ge \frac{\lvert c \rvert d \delta^{\mu - \mu_{r(i)}}}{i^{(D+1+\varepsilon)\mu_{r(i)}}} \ge \frac{\lvert c \rvert c d \delta^\mu}{i^{(D+1 + \varepsilon) \mu_0}}. \]
Since again \(m^+(p_k(z)_i) \le 2u\) for large \(i\), it follows that
\[ \liminf_{i \rightarrow \infty} \alpha(p_k(z)_i) \ge \frac{1}{(D+1+\varepsilon) \mu_0} = \frac{\alpha^{(2)}(A)}{(D+1+\varepsilon)}. \]
with arbitrary \(\varepsilon > 0\).  Lemma~\ref{lemma:sequencenet}, Proposition~\ref{prop:rbystoonebyone}, Proposition~\ref{prop:equalmatrices} and Proposition~\ref{prop:equallimits} finish the proof.
\end{proof}

\noindent Of course, this also completes the proof of Theorem~\ref{thm:liminfpositive}.

\end{document}